\documentclass[leqno,12pt]{article} 
\usepackage{amsmath, amssymb}
\usepackage{amsthm} 
\usepackage{pictexwd,dcpic}
\usepackage{multirow}
\theoremstyle{plain} 
\newtheorem{theorem}{\indent\sc Theorem}[section]
\newtheorem{lemma}[theorem]{\indent\sc Lemma}

\newtheorem{proposition}[theorem]{\indent\sc Proposition}

\theoremstyle{definition} 

\newtheorem{remark}[theorem]{\indent\sc Remark}
\newtheorem{example}[theorem]{\indent\sc Example}

\makeatletter
\def\address#1#2{\begingroup
\noindent\parbox[t]{7.8cm}{%
\small{\scshape\ignorespaces#1}\par\vskip1ex
\noindent\small{\itshape E-mail address}%
\/: #2\par\vskip4ex}\hfill%
\endgroup}%
\makeatother
\title{Generation of ring class fields by eta-quotients} 
\author{
\textsc{Ja Kyung Koo, Dong Hwa Shin and Dong Sung Yoon$^*$} 
}
\date{} 
\begin{document}

\maketitle

\footnote{$^*$Corresponding author.}
\footnote{
\thanks{The second named author was
supported by Hankuk University of Foreign Studies Research Fund of
2017.
} }

\begin{abstract}
We generate ring class fields of imaginary quadratic fields in terms
of the special values of certain eta-quotients, which are related to
the relative norms of Siegel-Ramachandra invariants. These give us
minimal polynomials with relatively small coefficients from which we are able to solve certain
quadratic Diophantine equations concerning non-convenient numbers.
\end{abstract}

\maketitle

\section {Introduction}

Let $K$ be an imaginary quadratic field of discriminant $d_K$. We set
\begin{equation}\label{tau_K}
\tau_K=\left\{\begin{array}{ll}
(-1+\sqrt{d_K})/2 & \textrm{if}~d_K\equiv1\pmod{4},\\
\sqrt{d_K}/2 & \textrm{if}~d_K\equiv0\pmod{4},
\end{array}\right.
\end{equation}
which belongs to the complex upper half-plane $\mathbb{H}$ and
generates the ring of integers $\mathcal{O}_K$ of $K$ over
$\mathbb{Z}$. Let $\mathcal{O}=[N\tau_K,1]$ be the order of
conductor $N$ ($\geq1$) in $K$. As a consequence of the main theorem
of complex multiplication, the singular value
$j(\mathcal{O})=j(N\tau_K)$ of the elliptic modular function
$j(\tau)$ generates the ring class field $H_\mathcal{O}$ of the
order $\mathcal{O}$ over $K$ (\cite[Chapter 10, Theorem 5]{Lang} or
\cite[Theorem 5.7]{Shimura}). However, its minimal polynomial has
too large integer coefficients to handle in practical use (see
\cite[Theorem 9.2 and $\S$13.A]{Cox}).
\par
The \textit{Dedekind eta-function} and \textit{modular discriminant function} are defined by
\begin{equation}\label{etaDelta}
\eta(\tau)=q^{1/24}\prod_{n=1}^\infty(1-q^n)
\quad\textrm{and}\quad\Delta(\tau)=(2\pi)^{12}\eta(\tau)^{24}\quad(\tau\in\mathbb{H},~q=e^{2\pi i\tau}),
\end{equation}
respectively.

\par
Let $N$ ($\geq2$) be an integer with
prime factorization $N=\prod_{k=1}^m p_k^{r_k}$.
For each subset $S$ of $\{1,\ldots,m\}$ we set
\begin{equation}\label{P_S}
P_S=\left\{\begin{array}{ll}
1 & \textrm{if}~S=\emptyset,\\
\prod_{k\in S} p_k & \textrm{if}~S\neq\emptyset,
\end{array}\right.
\end{equation}
and let
\begin{equation}\label{munu}
P_N=\prod_{k=1}^m(p_k-1),\quad
\nu_N=\left\{\begin{array}{ll}p_1&\textrm{if}~m=1,\\
1 & \textrm{if}~m\geq2,\end{array}\right.\quad
\mu_N=\left\{\begin{array}{ll}2&\textrm{if}~m=1~\textrm{and}~p_1\equiv1\pmod{8},\\
1 & \textrm{otherwise}.\end{array}\right.
\end{equation}
And, for each nontrivial ideal $\mathfrak{f}$ of $\mathcal{O}_K$ we consider
the canonical homomorphism
\begin{equation*}
\pi_\mathfrak{f}:\mathcal{O}_K\rightarrow\mathcal{O}_K/\mathfrak{f}.
\end{equation*}
In this paper we shall prove that if the order of the
group
\begin{equation*}
\pi_{p_k^{r_k}\mathcal{O}_K}(\mathcal{O}_K)^\times/
\pi_{p_k^{r_k}\mathcal{O}_K}(\mathcal{O}_K^\times)
\pi_{p_k^{r_k}\mathcal{O}_K}(\mathbb{Z})^\times
\end{equation*}
is greater than $2$ for each $k\in\{1,\ldots,m\}$, then
the special value
\begin{equation*}
\nu_N^{12\mu_N/\gcd(24,P_N)}
\prod_{S\subseteq\{1,\ldots,m\}}\eta((N/P_S)\tau_K)^{24(-1)^{|S|}\mu_N/\gcd(24,P_N)}
\end{equation*}
generates $H_\mathcal{O}$ over $K$ (Theorem \ref{main} and Remark
\ref{condition}). This ring class invariant is a real algebraic
integer whose minimal polynomial has relatively small coefficients
(Examples \ref{Example1}, \ref{Example2}, \ref{Example3}). To this
end we shall further improve Schertz's idea \cite{Schertz2} on
characters of class groups (Lemmas \ref{divide} and
\ref{constructcharacter}) when utilizing the second Kronecker limit
formula concerning Siegel-Ramachandra invariants (Proposition
\ref{Kronecker}).

\section {Eta-quotients}

We shall show that the
special values of certain $\eta$-quotients lie in the ring class fields
of imaginary quadratic fields.

\begin{lemma}\label{modularity}
Let $N$ \textup{($\geq2$)} be an integer.
Assume that a family of integers $\{m_d\}_{d|N}$, where $d$ runs over all positive divisors of $N$, satisfies the following conditions:
\begin{itemize}
\item[\textup{(i)}] $\sum_{d|N}m_d=0$.
\item[\textup{(ii)}] $\sum_{d|N}dm_d\equiv\sum_{d|N}(N/d)m_d\equiv0\pmod{24}$.
\item[\textup{(iii)}] $\prod_{d|N}d^{m_d}$
is a square in $\mathbb{Q}$.
\end{itemize}
Then the $\eta$-quotient
\begin{equation*}
f(\tau)=\prod_{d|N}\eta(d\tau)^{m_d}
\end{equation*}
is a meromorphic modular function on $\Gamma_0(N)=\{\gamma\in\mathrm{SL}_2(\mathbb{Z})
~|~\gamma\equiv\left[\begin{matrix}*&*\\0&*\end{matrix}\right]\pmod{N}\}$.
\end{lemma}
\begin{proof}
See \cite[Theorem 1.64]{Ono}.
\end{proof}

\begin{remark}
By the definition (\ref{etaDelta}) we see the following identity
\begin{equation*}
f(\tau)=q^{(1/24)\sum_{d|N}dm_d}\prod_{d|N}\prod_{n=1}^\infty(1-q^{dn})^{m_d}.
\end{equation*}
So, $f(\tau)$ has rational Fourier coefficients with respect to $q$ and
has neither zeros nor poles on $\mathbb{H}$.
\end{remark}

\begin{proposition}\label{etaquotient}
Let $N$ \textup{($\geq2$)} be an integer with prime factorization
$N=\prod_{k=1}^m p_k^{r_k}$.
With the same notations as in \textup{(\ref{P_S})} and \textup{(\ref{munu})}, the $\eta$-quotient
\begin{equation*}
g(\tau)=\prod_{S\subseteq\{1,\ldots,m\}}\eta((N/P_S)\tau)^{24(-1)^{|S|}\mu_N/\gcd(24,P_N)}
\end{equation*}
is a weakly holomorphic modular function on $\Gamma_0(N)$ with rational Fourier coefficients.
\end{proposition}
\begin{proof}
By Lemma \ref{modularity} it suffices to show that
\begin{itemize}
\item[(i)] $\sum_{S\subseteq\{1,\ldots,m\}}(-1)^{|S|}=0$,
\item[(ii)] $\sum_{S\subseteq\{1,\ldots,m\}}(N/P_S)(-1)^{|S|}\equiv
\sum_{S\subseteq\{1,\ldots,m\}}(P_S)(-1)^{|S|}\equiv
0\pmod{P_N}$,
\item[(iii)] $\prod_{S\subseteq\{1,\ldots,m\}}(N/P_S)^{24(-1)^{|S|}\mu_N/\gcd(24,P_N)}$ is a square in $\mathbb{Q}$.
\end{itemize}
\par
First, we obtain that
\begin{equation*}
\sum_{S\subseteq\{1,\ldots,m\}}(-1)^{|S|}=
\sum_{\ell=0}^m\sum_{|S|=\ell}(-1)^{|S|}
=\sum_{\ell=0}^m\dbinom{m}{\ell}(-1)^\ell=0.
\end{equation*}
\par
Next, we derive that
\begin{eqnarray*}
\sum_{S\subseteq\{1,\ldots,m\}}(N/P_S)(-1)^{|S|}&=&\sum_{\ell=0}^m\sum_{|S|=\ell}
(N/P_S)(-1)^{|S|}\\
&=&N\sum_{\ell=0}^m(-1)^\ell\sum_{|S|=\ell}1/P_S\\
&=&N\prod_{k=1}^m(1-1/p_k)\\
&=&(N/p_1p_2\cdots p_m)\prod_{k=1}^m(p_k-1)\\
&\equiv&0\pmod{P_N},
\end{eqnarray*}
and in a similar way we get
\begin{equation*}
\sum_{S\subseteq\{1,\ldots,m\}}P_S(-1)^{|S|}=\sum_{\ell=0}^m\sum_{|S|=\ell}P_S(-1)^{|S|}
=\sum_{\ell=0}^m(-1)^\ell\sum_{|S|=\ell}P_S
=\prod_{k=1}^m(1-p_k)
\equiv0\pmod{P_N}.
\end{equation*}
\par
Lastly, we deduce that
\begin{eqnarray*}
\prod_{S\subseteq\{1,\ldots,m\}}(N/P_S)^{24(-1)^{|S|}\mu_N/\gcd(24,P_N)}
&=&\prod_{S\subseteq\{1,\ldots,m\}}P_S^{-24(-1)^{|S|}\mu_N/\gcd(24,P_N)}\quad\textrm{by (i)}\\
&=&
(p_1\cdots p_m)^{(\sum_{\ell=0}^{m-1}\left(\begin{smallmatrix}m-1\\\ell\end{smallmatrix}\right)(-1)^\ell)24\mu_N/\gcd(24,P_N)}\\
&=&\left\{\begin{array}{ll}
p_1^{24\mu_N/\gcd(24,p_1-1)}& \textrm{if}~m=1,\\
1 & \textrm{if}~m\geq2,
\end{array}\right.
\end{eqnarray*}
which is a square in $\mathbb{Q}$ by the definition of $\mu_N$ in (\ref{munu}).
This proves the proposition.
\end{proof}

For a positive integer $N$ let $\mathcal{F}_N$ be the field of
meromorphic modular functions on
$\Gamma(N)=\{\gamma\in\mathrm{SL}_2(\mathbb{Z}) ~|~\gamma\equiv
I_2\pmod{N}\}$ whose Fourier coefficients with respect to $q^{1/N}$
lie in $\mathbb{Q}(\zeta_N)$, where $\zeta_N=e^{2\pi i/N}$. Then it
is well-known that $\mathcal{F}_N$ is a Galois extension of
$\mathcal{F}_1=\mathbb{Q}(j(\tau))$ whose Galois group is isomorphic
to
\begin{equation*}
\mathrm{GL}_2(\mathbb{Z}/N\mathbb{Z})/\{\pm I_2\}=
\left\{\left[\begin{matrix}1&0\\0&d\end{matrix}\right]~|~d\in(\mathbb{Z}/N\mathbb{Z})^\times\right\}
\cdot\mathrm{SL}_2(\mathbb{Z}/N\mathbb{Z})/\{\pm I_2\}.
\end{equation*}
Let $h(\tau)=\sum_{n>-\infty}c_nq^{n/N}\in\mathcal{F}_N$ with $c_n\in\mathbb{Q}(\zeta_N)$.
The matrix $\left[\begin{matrix}1&0\\0&d\end{matrix}\right]$ with $d\in(\mathbb{Z}/N\mathbb{Z})^\times$ acts on
$h(\tau)$ by
\begin{equation*}
h(\tau)^{\left[\begin{smallmatrix}1&0\\0&d\end{smallmatrix}\right]}=
\sum_{n>-\infty}c_n^{\sigma_d}q^{n/N},
\end{equation*}
where $\sigma_d$ is the automorphism of $\mathbb{Q}(\zeta_N)$
given by $\zeta_N^{\sigma_d}=\zeta_N^d$.
Now, let $\alpha\in\mathrm{SL}_2(\mathbb{Z}/N\mathbb{Z})/\{\pm I_2\}$.
Take a preimage $\widetilde{\alpha}\in\mathrm{SL}_2(\mathbb{Z})$ of $\alpha$
with respect to the reduction $\mathrm{SL}_2(\mathbb{Z})\rightarrow\mathrm{SL}_2(\mathbb{Z}/N\mathbb{Z})/\{\pm I_2\}$.
Then, $\alpha$ acts on $h(\tau)$ by the fractional linear transformation of $\widetilde{\alpha}$ \cite[Chapter 6, Theorem 3]{Lang}.
\par
Throughout this paper we let $K$ be an imaginary quadratic field of
discriminant $d_K$ with ring of integers $\mathcal{O}_K$ and
$\tau_K$ be as in (\ref{tau_K}). For a nontrivial ideal
$\mathfrak{f}$ of $\mathcal{O}_K$ we denote by
\begin{eqnarray*}
I_K(\mathfrak{f})&=&\textrm{the group of fractional ideals of $K$ prime to $\mathfrak{f}$},\\
P_{K,1}(\mathfrak{f})&=&\langle x\mathcal{O}_K~|~x\in\mathcal{O}_K~\textrm{such that}~x\equiv1\pmod{\mathfrak{f}}\rangle.
\end{eqnarray*}
In particular, if $\mathfrak{f}=N\mathcal{O}_K$ for a positive integer $N$, then we further denote by
\begin{equation*}
P_{K,\mathbb{Z}}(\mathfrak{f})=\langle x\mathcal{O}_K~|~x\in\mathcal{O}_K~\textrm{such that}
~x\equiv n\pmod{\mathfrak{f}}~\textrm{for some $n\in\mathbb{Z}$ prime to $N$}\rangle.
\end{equation*}
By the existence theorem of class field theory there exists a unique abelian extension $K_\mathfrak{f}$ of $K$,
called the \textit{ray class field} of $K$ modulo $\mathfrak{f}$,
whose Galois group $\mathrm{Gal}(K_\mathfrak{f}/K)$ is isomorphic to
the ray class group $\mathrm{Cl}(\mathfrak{f})=I_K(\mathfrak{f})/P_{K,1}(\mathfrak{f})$
modulo $\mathfrak{f}$ via the Artin reciprocity map
$\sigma_\mathfrak{f}:\mathrm{Cl}(\mathfrak{f})\rightarrow\mathrm{Gal}(K_\mathfrak{f}/K)$
\cite[Chapters IV and V]{Janusz}.
In particular, if $\mathfrak{f}=\mathcal{O}_K$, then $K_\mathfrak{f}$ becomes the Hilbert
class field $H_K$ of $K$, that is, the maximal unramified abelian extension of $K$.
\par
Now, we let $\mathfrak{f}=N\mathcal{O}_K$ for a positive integer $N$.
As a consequence of the main theorem of complex multiplication we
obtain
\begin{equation}\label{CM}
H_K=K(j(\tau_K))\quad\textrm{and}\quad
K_\mathfrak{f}=K(h(\tau_K)~|~h(\tau)\in\mathcal{F}_N~\textrm{is finite at}~\tau_K)
\end{equation}
\cite[Chapter 10, Theorem 1 and Corollary to Theorem 2]{Lang}.
Let $\min(\tau_K,\mathbb{Q})=X^2+bX+c$ and define a subgroup $W_{K,N}$ of $\mathrm{GL}_2(\mathbb{Z}/N\mathbb{Z})$ by
\begin{equation*}
W_{K,N}=\left\{\left[\begin{matrix}t-bs & -cs\\s&t\end{matrix}\right]\in\mathrm{GL}_2(\mathbb{Z}/N\mathbb{Z})~|~t,s\in\mathbb{Z}/N\mathbb{Z}\right\}.
\end{equation*}
By Shimura's reciprocity law we
have the surjection
\begin{eqnarray*}
W_{K,N}&\rightarrow&\mathrm{Gal}(K_\mathfrak{f}/H_K)\\
\gamma&\mapsto&(h(\tau_K)\mapsto h^\gamma(\tau_K)~|~h(\tau)\in\mathcal{F}_N~\textrm{is finite at}~\tau_K)
\end{eqnarray*}
whose kernel is
\begin{equation*}
T_{K,N}=\left\{
\begin{array}{ll}
\langle\left[\begin{smallmatrix}0&-1\\1&0\end{smallmatrix}
\right]\rangle & \textrm{if}~K=\mathbb{Q}(\sqrt{-1}),\\
\langle
\left[\begin{smallmatrix}1&1\\-1&0\end{smallmatrix}\right]\rangle
& \textrm{if}~K=\mathbb{Q}(\sqrt{-3}),\\
\langle-I_2\rangle & \textrm{if}~K\neq\mathbb{Q}(\sqrt{-1}),\mathbb{Q}(\sqrt{-3})
\end{array}\right.
\end{equation*}
(\cite[Theorem 6.31 and Proposition 6.34]{Shimura} and \cite[$\S3$]{Stevenhagen}).
\par
On the other hand, let $\mathcal{O}=[N\tau_K,1]$ be the order of conductor $N$ in $K$. Then the ideal class group $\mathrm{C}(\mathcal{O})=I(\mathcal{O})/P(\mathcal{O})$ of the order $\mathcal{O}$, where $I(\mathcal{O})$ is the group of proper fractional $\mathcal{O}$-ideals and $P(\mathcal{O})$ is the subgroup of principal $\mathcal{O}$-ideals, is isomorphic to $I_K(\mathfrak{f})/P_{K,\mathbb{Z}}(\mathfrak{f})$
\cite[$\S$7.A and Proposition 7.22]{Cox}.
The \textit{ring class field} $H_\mathcal{O}$ of the order $\mathcal{O}$ is defined to
be the unique abelian extension of $K$ whose Galois group satisfies
\begin{equation}\label{GalH_O}
\mathrm{Gal}(H_\mathcal{O}/K)\simeq
\mathrm{C}(\mathcal{O})\simeq I_K(\mathfrak{f})/P_{K,\mathbb{Z}}(\mathfrak{f}).
\end{equation}

\begin{lemma}\label{Shimura}
We have the
isomorphism
\begin{eqnarray*}
\langle T_{K,N},tI_2~|~t\in(\mathbb{Z}/N\mathbb{Z})^\times\rangle/T_{K,N}
&\stackrel{\sim}{\rightarrow}&\mathrm{Gal}(K_\mathfrak{f}/H_\mathcal{O})\\
tI_2&\mapsto&(h(\tau_K)\mapsto h^{tI_2}(\tau_K)~|~h(\tau)\in\mathcal{F}_N~\textrm{is finite at}~\tau_K).
\end{eqnarray*}
\end{lemma}
\begin{proof}
See \cite[Proposition 5.3]{K-S} or \cite[Proposition 3.8]{E-K-S}.
\end{proof}

\begin{remark}\label{Galoisgroup}
Thus we obtain that
\begin{eqnarray*}
\mathrm{Gal}(H_\mathcal{O}/H_K)&\simeq&
\mathrm{Gal}(K_\mathfrak{f}/H_K)/\mathrm{Gal}(K_\mathfrak{f}/H_\mathcal{O})\\
&\simeq&
(W_{K,N}/T_{K,N})/(\langle T_{K,N},tI_2~|~t\in(\mathbb{Z}/N\mathbb{Z})^\times\rangle/T_{K,N})\\
&\simeq&
 W_{K,N}/\langle T_{K,N},tI_2~|~t\in(\mathbb{Z}/N\mathbb{Z})^\times\rangle.
\end{eqnarray*}
\end{remark}

\begin{lemma}\label{liein}
Let $h(\tau)$ be a meromorphic modular function on $\Gamma_0(N)$ with rational
Fourier coefficients. If $h(\tau)$ is finite at $\tau_K$,
then $h(\tau_K)$ lies in $H_\mathcal{O}$.
\end{lemma}
\begin{proof}
Since $h(\tau_K)$ belongs to $K_\mathfrak{f}$ by (\ref{CM}), it suffices to show by Lemma \ref{Shimura} that every $tI_2$ with $t\in(\mathbb{Z}/N\mathbb{Z})^\times$ leaves $h(\tau_K)$ fixed. Decompose $tI_2$ into $tI_2=\left[\begin{matrix}1&0\\0&t^2\end{matrix}\right]
\left[\begin{matrix}t&0\\0&t^{-1}\end{matrix}\right]$
in $\mathrm{GL}_2(\mathbb{Z}/N\mathbb{Z})$, and let $\widetilde{\alpha}\in\mathrm{SL}_2(\mathbb{Z})$ be a preimage of $\left[\begin{matrix}
t&0\\0&t^{-1}\end{matrix}\right]\in\mathrm{SL}_2(\mathbb{Z}/N\mathbb{Z})$ with respect to
the reduction $\mathrm{SL}_2(\mathbb{Z})\rightarrow\mathrm{SL}_2(\mathbb{Z}/N\mathbb{Z})/\{\pm I_2\}$.
Then we know $\widetilde{\alpha}\in\Gamma_0(N)$ and achieve that
\begin{eqnarray*}
h(\tau_K)^{tI_2}&=&h^{tI_2}(\tau_K)\quad\textrm{by Lemma \ref{Shimura}}\\
&=&h^{\left[\begin{smallmatrix}1&0\\0&t^2\end{smallmatrix}\right]\widetilde{\alpha}}(\tau_K)\\
&=&h^{\widetilde{\alpha}}(\tau_K)\quad\textrm{since $h(\tau)$ has rational Fourier coefficients}\\
&=&h(\tau_K)\quad\textrm{because}~\widetilde{\alpha}\in\Gamma_0(N).
\end{eqnarray*}
This proves the lemma.
\end{proof}

\begin{proposition}\label{etaliein}
Let $N\geq2$ be an integer and $g(\tau)$ be the $\eta$-quotient described in
\textup{Proposition \ref{etaquotient}}. Then the special value
$g(\tau_K)$ belongs to $H_\mathcal{O}$ as a real algebraic number.
\end{proposition}
\begin{proof}
By Proposition \ref{etaquotient} and Lemma \ref{liein}, $g(\tau_K)$ lies in $H_\mathcal{O}$. Furthermore, since $g(\tau)$ has rational Fourier coefficients with respect to $q$ and
$e^{2\pi i\tau_K}$ is a real number, $g(\tau_K)$ is a real algebraic number.
\end{proof}

\section {Characters on class groups}

Let $N$ be a positive integer, $\mathfrak{f}=N\mathcal{O}_K$ and
\begin{equation*}
\pi_\mathfrak{f}:\mathcal{O}_K\rightarrow\mathcal{O}_K/\mathfrak{f}
\end{equation*}
be the canonical homomorphism. Define group homomorphisms
\begin{eqnarray*}
\widetilde{\Phi}_\mathfrak{f}:\pi_\mathfrak{f}(\mathcal{O}_K)^\times&\rightarrow&\mathrm{Cl}(\mathfrak{f})\\
x+\mathfrak{f}&\mapsto&[x\mathcal{O}_K],~\textrm{the class containing}~x\mathcal{O}_K
\end{eqnarray*}
and
\begin{equation*}
\widetilde{\Psi}_\mathfrak{f}:\pi_\mathfrak{f}(\mathcal{O}_K)^\times\rightarrow
I_K(\mathfrak{f})/P_{K,\mathbb{Z}}(\mathfrak{f})
\end{equation*}
as the composition of $\widetilde{\Phi}_\mathfrak{f}$ and
the natural surjection $\mathrm{Cl}(\mathfrak{f})=I_K(\mathfrak{f})/P_{K,1}(\mathfrak{f})\rightarrow I_K(\mathfrak{f})/
P_{K,\mathbb{Z}}(\mathfrak{f})$.

\begin{lemma}\label{Ker}
We have
\begin{equation*}
\mathrm{Ker}(\widetilde{\Psi}_\mathfrak{f})=\pi_\mathfrak{f}(\mathcal{O}_K^\times)\pi_\mathfrak{f}(\mathbb{Z})^\times.
\end{equation*}
\end{lemma}
\begin{proof}
We deduce that
\begin{eqnarray*}
&&x+\mathfrak{f}\in\pi_\mathfrak{f}(\mathcal{O}_K)^\times~\textrm{belongs to}~\mathrm{Ker}(\widetilde{\Psi}_\mathfrak{f})\\
&\Longleftrightarrow&x\mathcal{O}_K\in P_{K,\mathbb{Z}}(\mathfrak{f})\\
&\Longleftrightarrow&x\zeta\equiv n\pmod{\mathfrak{f}}~\textrm{for some}~\zeta\in\mathcal{O}_K^\times~\textrm{and}~n\in\mathbb{Z}~\textrm{prime to}~N\\
&\Longleftrightarrow&x+\mathfrak{f}=\zeta^{-1}n+\mathfrak{f}\\
&\Longleftrightarrow&x+\mathfrak{f}\in\pi_\mathfrak{f}(\mathcal{O}_K^\times)\pi_\mathfrak{f}(\mathbb{Z})^\times.
\end{eqnarray*}
This proves the lemma.
\end{proof}

\begin{lemma}\label{classical}
Let $G$ be a finite abelian group and
$H$ be a subgroup of $G$. Let $g\in G$ and $n$
be the smallest positive integer such that $g^n\in H$.
If $\chi$ is a character of $H$, then it can be extended
to a character $\chi'$ of $G$ for which $\chi'(g)$
is any $n$-th root of $\chi(g^n)$.
\end{lemma}
\begin{proof}
See \cite[Chapter VI]{Serre}.
\end{proof}

Let $\mathfrak{a}$ be any nontrivial ideal of $\mathcal{O}_K$ and
$\chi$ be a character of $(\mathcal{O}_K/\mathfrak{a})^\times$.
Recall that the conductor $\mathfrak{f}_\chi$ of $\chi$ is defined by
\begin{align}
&&\mathfrak{f}_\chi=\gcd\{\textrm{nontrivial ideals $\mathfrak{m}$ of $\mathcal{O}_K$}~|~\chi(\alpha+\mathfrak{a})=1~\textrm{for all $\alpha
\in\mathcal{O}_K$ such that}\label{conddef}\\
&&\textrm{$\alpha\mathcal{O}_K$ is prime to $\mathfrak{a}$
and $\alpha\equiv1\pmod{\mathfrak{m}}$}\}. \nonumber
\end{align}
In particular, $\mathfrak{f}_\chi$ divides $\mathfrak{a}$.

\begin{lemma}\label{divide}
Let $\mathfrak{g}=p^r\mathcal{O}_K$ for a prime $p$ and a positive
integer $r$. Let $\chi$ be a nonprincipal character of
$\pi_\mathfrak{g}(\mathcal{O}_K)^\times$ which is trivial on
$\pi_\mathfrak{g}(\mathbb{Z})^\times$. Then the conductor
$\mathfrak{f}_\chi$ of $\chi$ is divisible by every prime ideal
factor of $p\mathcal{O}_K$.
\end{lemma}
\begin{proof}
Since $\chi$ is nonprincipal, the assertion is obvious if $p$ ramifies or is inert in $K$.
\par
Now let $p$ split in $K$, so $p\mathcal{O}_K$ has prime ideal
factorization $p\mathcal{O}_K=\mathfrak{p}\overline{\mathfrak{p}}$
with $\mathfrak{p}\neq\overline{\mathfrak{p}}$. Without loss of
generality we may assume that $\mathfrak{p}$ divides
$\mathfrak{f}_\chi$. Suppose on the contrary that
$\overline{\mathfrak{p}}$ does not divide $\mathfrak{f}_\chi$, and
hence $\mathfrak{f}_\chi$ is a nonzero power of $\mathfrak{p}$. Let
$\alpha$ be any element of $\mathcal{O}_K$ such that
$\alpha\mathcal{O}_K$ is prime to $\mathfrak{g}$ and
$\alpha\equiv1\pmod{\overline{\mathfrak{p}}^{~r}}$. Since
$\overline{\alpha}\equiv1\pmod{\mathfrak{p}^{r}}$ and
$\mathfrak{f}_\chi$ divides $\mathfrak{p}^r$, we get by the definition (\ref{conddef})
\begin{equation}\label{bar}
\chi(\overline{\alpha}+\mathfrak{g})=1.
\end{equation}
On the other hand,
since $\mathrm{N}_{K/\mathbb{Q}}(\alpha)\in\mathbb{Z}$ and
$\chi$ is trivial on $\pi_\mathfrak{g}(\mathbb{Z})^\times$,
we find
\begin{eqnarray*}
\chi(\alpha+\mathfrak{g})\chi(\overline{\alpha}+\mathfrak{g})=\chi((\alpha+\mathfrak{g})
(\overline{\alpha}+\mathfrak{g}))
=\chi(\mathrm{N}_{K/\mathbb{Q}}(\alpha)+\mathfrak{g})
=1.
\end{eqnarray*}
It then follows from (\ref{bar}) that
\begin{equation*}
\chi(\alpha+\mathfrak{g})=1,
\end{equation*}
which implies that $\mathfrak{f}_\chi$ divides
$\overline{\mathfrak{p}}^{~r}$. But this contradicts the fact that
$\mathfrak{f}_\chi$ is a nonzero power of $\mathfrak{p}$, and so
$\mathfrak{f}_\chi$ is divisible by both $\mathfrak{p}$ and
$\overline{\mathfrak{p}}$.
\end{proof}

For any intermediate field $F$ of the extension $K_\mathfrak{f}/K$ we mean
by $\mathrm{Cl}(K_\mathfrak{f}/F)$
the subgroup of $\mathrm{Cl}(\mathfrak{f})$ ($\simeq\mathrm{Gal}(K_\mathfrak{f}/K)$)
corresponding to $\mathrm{Gal}(K_\mathfrak{f}/F)$.

\begin{lemma}\label{constructcharacter}
For an integer $N\geq2$ with prime factorization $N=\prod_{k=1}^m
p_k^{r_k}$,
let $\mathcal{O}$ be the order of conductor $N$ in $K$.
Let $C\in\mathrm{Cl}(\mathfrak{f})\setminus\mathrm{Cl}(K_\mathfrak{f}/H_\mathcal{O})$ \textup{(}if any\textup{)}.
Assume that the order of the group
\begin{equation*}
G_k=\pi_{p_k^{r_k}\mathcal{O}_K}(\mathcal{O}_K)^\times/
\pi_{p_k^{r_k}\mathcal{O}_K}(\mathcal{O}_K^\times)
\pi_{p_k^{r_k}\mathcal{O}_K}(\mathbb{Z})^\times
\end{equation*}
is greater than $2$ for each $k\in\{1,\ldots,m\}$.
Then there exists a character $\chi$ of $\mathrm{Cl}(\mathfrak{f})$ satisfying
\begin{itemize}
\item[\textup{(i)}] $\chi$ is trivial on $\mathrm{Cl}(K_\mathfrak{f}/H_\mathcal{O})$,
\item[\textup{(ii)}] $\chi(C)\neq1$,
\item[\textup{(iii)}] every prime ideal factor of $\mathfrak{f}$
divides the conductor $\mathfrak{f}_\chi$ of $\chi$.
\end{itemize}
\end{lemma}
\begin{proof}
By Lemma \ref{classical} we can take a character $\chi$ of
$\mathrm{Cl}(\mathfrak{f})$ satisfying (i) and (ii).
Observe that the
conductor $\mathfrak{f}_\chi$ of $\chi$ is defined to be that of the
character $\widetilde{\chi}=\chi\circ\widetilde{\Phi}_\mathfrak{f}$
of $\pi_\mathfrak{f}(\mathcal{O}_K)^\times$. By the Chinese
remainder theorem we have the natural isomorphism
\begin{equation*}
\pi_\mathfrak{f}(\mathcal{O}_K)^\times\simeq
\prod_{k=1}^m\pi_{p_k^{r_k}\mathcal{O}_K}(\mathcal{O}_K)^\times.
\end{equation*}
For each $k\in\{1,\ldots,m\}$, let
\begin{equation*}
\widetilde{\iota}_k:\pi_{p_k^{r_k}\mathcal{O}_K}(\mathcal{O}_K)^\times\hookrightarrow
\prod_{k=1}^m\pi_{p_k^{r_k}\mathcal{O}_K}(\mathcal{O}_K)^\times
\stackrel{\sim}{\rightarrow}\pi_\mathfrak{f}(\mathcal{O}_K)^\times
\end{equation*}
be the natural injection and $\widetilde{\chi}_k=\widetilde{\chi}\circ\widetilde{\iota}_k$.
Then it is obvious by the definition (\ref{conddef}) that $\mathfrak{f}_\chi$ is divisible by
$\prod_{k=1}^m\mathfrak{f}_{\widetilde{\chi}_k}$.
Let $\mathcal{O}_k$ be the order of conductor $p_k^{r_k}$ in $K$ (hence, $H_{\mathcal{O}_k}\subseteq H_\mathcal{O}$).
By Lemma \ref{Ker} and (\ref{GalH_O}) we obtain an injection
\begin{align*}
&&G_k
\rightarrow
I_K(p_k^{r_k}\mathcal{O}_K)/P_{K,\mathbb{Z}}(p_k^{r_k}\mathcal{O}_K)
\stackrel{\sim}{\rightarrow}\mathrm{Gal}(H_{\mathcal{O}_k}/K)\stackrel{\sim}{\rightarrow}
\mathrm{Gal}(K_\mathfrak{f}/K)/\mathrm{Gal}(K_\mathfrak{f}/H_{\mathcal{O}_k})\\
&&\stackrel{\sim}{\rightarrow}
\mathrm{Cl}(\mathfrak{f})/\mathrm{Cl}(K_\mathfrak{f}/H_{\mathcal{O}_k}).
\end{align*}
Thus we may identify $G_k$ with
a subgroup of $\mathrm{Cl}(\mathfrak{f})/\mathrm{Cl}(K_\mathfrak{f}/H_{\mathcal{O}_k})$.
\par
Suppose that $\widetilde{\chi}_k=1$ for some $k\in\{1,\ldots,m\}$.
If we let $e_k$ be the exponent of $G_k$, then we are faced with two possible cases.
\begin{itemize}
\item[Case 1.] First, consider the case $e_k=2$.
Let $[C]$ be the class of $C$ in $\mathrm{Cl}(\mathfrak{f})/
\mathrm{Cl}(K_\mathfrak{f}/H_{\mathcal{O}_k})$ and
$n$ be the smallest positive integer so that $[C]^n$ belongs to
$G_k$. Since $e_k=2$ and $|G_k|>2$ by hypothesis, $G_k$ is not a cyclic group. In particular, $G_k\supsetneq\langle[C]^n\rangle$. Hence there
exists a nonprincipal character $\psi$
of $G_k$ such that $\psi([C]^n)=1$
by Lemma \ref{classical}. And, we can extend $\psi$ to a character $\psi'$ of $\mathrm{Cl}(\mathfrak{f})/\mathrm{Cl}(K_\mathfrak{f}/H_{\mathcal{O}_k})$
in such a way that $\psi'([C])=1$ again by Lemma \ref{classical}.
Let
\begin{equation*}
\psi''=\psi'\circ(\mathrm{Cl}(\mathfrak{f})\rightarrow\mathrm{Cl}(\mathfrak{f})/
\mathrm{Cl}(K_\mathfrak{f}/H_{\mathcal{O}_k})),
\end{equation*}
which is trivial on $\mathrm{Cl}(K_\mathfrak{f}/H_{\mathcal{O}_k})$ and $\psi''(C)=1$.
We then see that $\chi\psi''$ is trivial on
$\mathrm{Cl}(K_\mathfrak{f}/H_{\mathcal{O}_k})$ (so, on $\mathrm{Cl}(K_\mathfrak{f}/H_{\mathcal{O}})$) and
\begin{equation*}
(\chi\psi'')(C)=\chi(C)\psi''(C)=
\chi(C)\neq1.
\end{equation*}
Moreover, since $\widetilde{\chi}_k=1$, we get
\begin{equation*}
(\chi\psi'')\circ\widetilde{\Phi}_\mathfrak{f}\circ\widetilde{\iota}_k
=\psi
\circ(\pi_{p_k^{r_k}\mathcal{O}_K}(\mathcal{O}_K)^\times
\rightarrow G_k),
\end{equation*}
which is trivial on $\pi_{p_k^{r_k}\mathcal{O}_K}(\mathcal{O}_K^\times)
\pi_{p_k^{r_k}\mathcal{O}_K}(\mathbb{Z})^\times$
as a nonprincipal character of $\pi_{p_k^{r_k}\mathcal{O}_K}(\mathcal{O}_K)^\times$.
Therefore every prime
ideal factor of $p_k\mathcal{O}_K$ divides the
conductor $\mathfrak{f}_{\chi\psi''}$ of $\chi\psi''$ by Lemma
\ref{divide}. And, we further note that
\begin{equation*}
(\chi\psi'')\circ
\widetilde{\Phi}_\mathfrak{f}\circ\widetilde{\iota}_\ell
=\widetilde{\chi}_\ell\quad\textrm{for all $\ell\in\{1,\ldots,m\}$ such that $\ell\neq k$}
\end{equation*}
by the construction of $\psi''$.
Now, we replace $\chi$ by $\chi\psi''$.
\item[Case 2.] Next, let $e_k>2$.
Then there is a character $\xi$ of $G_k$ such that
$\xi^2\neq1$, because the exponent of the character group of $G_k$ is also
greater than $2$. Extend $\xi$ to a character $\xi'$ of
$\mathrm{Cl}(\mathfrak{f})/\mathrm{Cl}(K_\mathfrak{f}/H_{\mathcal{O}_k})$
by using Lemma \ref{classical}, and let
\begin{equation*}
\xi''=\xi'\circ(\mathrm{Cl}(\mathfrak{f})\rightarrow\mathrm{Cl}(\mathfrak{f})/
\mathrm{Cl}(K_\mathfrak{f}/H_{\mathcal{O}_k})).
\end{equation*}
Then we see that $\xi''$ is trivial on
$\mathrm{Cl}(K_\mathfrak{f}/H_{\mathcal{O}_k})$ and $\xi''^2\neq1$.
Set
\begin{equation*}
\rho=\left\{\begin{array}{ll}
\xi'' & \textrm{if}~\xi''(C)\neq\chi(C)^{-1},\\
\xi''^2 & \textrm{if}~\xi''(C)=\chi(C)^{-1}~(\neq1).
\end{array}\right.
\end{equation*}
As in the above, one can readily check that
 $\chi\rho$ is trivial on
$\mathrm{Cl}(K_\mathfrak{f}/H_{\mathcal{O}})$,
$(\chi\rho)(C)\neq1$,
every prime factor of $p_k\mathcal{O}_K$ divides the conductor $\mathfrak{f}_{\chi\rho}$ of $\chi\rho$ and
\begin{equation*}
(\chi\rho)\circ
\widetilde{\Phi}_\mathfrak{f}\circ\widetilde{\iota}_\ell
=\widetilde{\chi}_\ell\quad\textrm{for all $\ell\in\{1,\ldots,m\}$ such that $\ell\neq k$}.
\end{equation*}
Hence, we replace $\chi$ by $\chi\rho$.
\end{itemize}
Continuing this way we eventually obtain a character $\chi$ of
$\mathrm{Cl}(\mathfrak{f})$ satisfying the conditions (i)$\sim$(iii)
in the lemma.
\end{proof}

\section {Ring class invariants}

For a vector $\left[\begin{matrix}r_1\\r_2\end{matrix}\right]\in\mathbb{Q}^2\setminus\mathbb{Z}^2$
we define the \textit{Siegel function} $g_{\left[\begin{smallmatrix}r_1\\r_2\end{smallmatrix}\right]}(\tau)$ on $\mathbb{H}$ by the following infinite product
\begin{eqnarray*}
g_{\left[\begin{smallmatrix}r_1\\r_2\end{smallmatrix}\right]}(\tau)=
-q^{(1/2)(r_1^2-r_1+1/6)}e^{\pi
ir_2(r_1-1)}(1-q^{r_1}e^{2\pi ir_2}) \prod_{n=1}^{\infty}(1-q^{n+r_1}e^{2\pi ir_2})(1-q^{n-r_1}e^{-2\pi ir_2}).
\end{eqnarray*}
If $M$ ($\geq2$) is an integer so that $Mr_1,Mr_2\in\mathbb{Z}$, then
$g_{\left[\begin{smallmatrix}r_1\\r_2\end{smallmatrix}\right]}(\tau)^{12M}$ belongs to
$\mathcal{F}_M$ and has neither zeros nor poles on $\mathbb{H}$ \cite[Chapter 2, Theorem 1.2]{K-L}.
\par
Let $\mathfrak{f}$ be a nontrivial proper ideal of $\mathcal{O}_K$ and $C\in\mathrm{Cl}(\mathfrak{f})$.
Let $N=N(\mathfrak{f})$ ($\geq2$) be the
smallest positive integer in $\mathfrak{f}$.
Take any integral ideal $\mathfrak{c}\in C$ and let
\begin{eqnarray*}
\mathfrak{f}\mathfrak{c}^{-1}&=&[\omega_1,\omega_2]\quad\textrm{for some}~ \omega_1,\omega_2\in\mathbb{C}~\textrm{such that}~\omega_1/\omega_2\in\mathbb{H},\\
1&=&(a/N)\omega_1+(b/N)\omega_2\quad\textrm{for some}~a,b\in\mathbb{Z}.
\end{eqnarray*}
We define the \textit{Siegel-Ramachandra invariant}
$g_\mathfrak{f}(C)$ of conductor $\mathfrak{f}$ at the class $C$ by
\begin{equation*}
g_\mathfrak{f}(C)=g_{\left[\begin{smallmatrix}a/N\\b/N\end{smallmatrix}\right]}(\omega_1/\omega_2)^{12N},
\end{equation*}
which depends only on $\mathfrak{f}$ and $C$, not on the choice of
$\mathfrak{c}$, $\omega_1$ and $\omega_2$ \cite[Chapter 11,
$\S$1]{K-L}. It lies in $K_\mathfrak{f}$ and satisfies the
transformation formula
\begin{equation}\label{transformation}
g_\mathfrak{f}(C)^{\sigma_\mathfrak{f}(C')}=g_\mathfrak{f}(CC')\quad\textrm{for any}~C'\in\mathrm{Cl}(\mathfrak{f}),
\end{equation}
where
$\sigma_\mathfrak{f}:\mathrm{Cl}(\mathfrak{f})\rightarrow\mathrm{Gal}(K_\mathfrak{f}/K)$
is the Artin reciprocity map \cite[Chapter 11, Theorem 1.1]{K-L}.
Ramachandra \cite{Ramachandra} further showed that
$g_\mathfrak{f}(C)$ is an algebraic integer and a unit if
$\mathfrak{f}$ is not a power of one prime ideal (see also
\cite[$\S$3]{K-S2}).
\par
For a nonprincipal character $\chi$ of $\mathrm{Cl}(\frak{f})$ we define the \textit{Stickelberger
element} and the \textit{$L$-function} as
\begin{equation*}
S_\mathfrak{f}(\chi)
=\sum_{C\in\mathrm{Cl}(\mathfrak{f})}
\chi(C)\ln|g_\mathfrak{f}(C)|
\quad\textrm{and}\quad L_\mathfrak{f}(s,\chi)=\hspace{-0.5cm}\sum_{\begin{smallmatrix}\mathfrak{a}~:~\textrm{nontrivial
ideals of
$\mathcal{O}_K$}\\
\textrm{prime to $\mathfrak{f}$}\end{smallmatrix}}\frac{\chi([\mathfrak{a}])}{\mathrm{N}_{K/\mathbb{Q}}(\mathfrak{a})^s}
\quad(s\in\mathbb{C}),
\end{equation*}
respectively.

\begin{proposition}[The second Kronecker limit formula]\label{Kronecker}
Let $\mathfrak{f}_\chi$ be the conductor of $\chi$ and $\chi_0$ be the proper character
of $\mathrm{Cl}(\mathfrak{f}_\chi)$ corresponding to $\chi$.
If $\mathfrak{f}_\chi\neq\mathcal{O}_K$, then we have
\begin{equation*}
L_{\mathfrak{f}_\chi}(1,\chi_0)
\prod_{\mathfrak{p}|\mathfrak{f},~\mathfrak{p}\nmid~\mathfrak{f}_\chi}
(1-\overline{\chi}_0([\mathfrak{p}]))=-
\frac{\pi\chi_0([\gamma\mathfrak{d}_K\mathfrak{f}_\chi])}{3N(\mathfrak{f}_\chi)\sqrt{-d_K}\omega(\mathfrak{f}_\chi)
T_\gamma(\overline{\chi}_0)}~S_{\mathfrak{f}}(\overline{\chi}),
\end{equation*}
where $\mathfrak{d}_K$ is
the different of $K/\mathbb{Q}$, $\gamma$ is an element of $K$ so that
$\gamma\mathfrak{d}_K\mathfrak{f}_\chi$ becomes an integral ideal of $K$ prime to $\mathfrak{f}_\chi$,
$N(\mathfrak{f}_\chi)$ is the smallest positive integer in $\mathfrak{f}_\chi$,
$\omega(\mathfrak{f}_\chi)
=|\{\zeta\in\mathcal{O}_K^\times~|~\zeta\equiv1\pmod{\mathfrak{f}_\chi}\}|$ and
\begin{eqnarray*}
T_\gamma(\overline{\chi}_0)=
\sum_{x+\mathfrak{f}_\chi\in\pi_{\mathfrak{f}_\chi}(\mathcal{O}_K)^\times}
\overline{\chi}_0([x\mathcal{O}_K])e^{2\pi
i\mathrm{Tr}_{K/\mathbb{Q}}(x\gamma)}.
\end{eqnarray*}
\end{proposition}
\begin{proof}
See \cite[Chapter 22, Theorems 1 and 2]{Lang} and \cite[Chapter 11, Theorem
2.1]{K-L}.
\end{proof}

\begin{remark}\label{Stickremark}
\begin{itemize}
\item[(i)] If every prime ideal factor of $\mathfrak{f}$ divides $\mathfrak{f}_\chi$,
then we understand the Euler factor $\prod_{\mathfrak{p}|\mathfrak{f},~\mathfrak{p}\nmid~\mathfrak{f}_\chi}
(1-\overline{\chi}_0([\mathfrak{p}]))$ to be $1$.
\item[(ii)] Since $\chi_0$ is a nonprincipal character of $\mathrm{Cl}(\mathfrak{f}_\chi)$, we get $L_{\mathfrak{f}_\chi}(1,\chi_0)\neq0$
\cite[Chapter V, Theorem 10.2]{Janusz}.
\item[(iii)] The Gauss sum $T_\gamma(\overline{\chi}_0)$ is nonzero, in particular, $|T_\gamma(\overline{\chi}_0)|=\sqrt{\mathrm{N}_{K/\mathbb{Q}}(\mathfrak{f}_\chi)}$
\cite[Chapter 22, $\S$1]{Lang}.
\end{itemize}
\end{remark}

Now, let $\mathfrak{f}=N\mathcal{O}_K$ for some integer $N$
($\geq2$) with prime factorization $N=\prod_{k=1}^m p_k^{r_k}$. Let
$\mathcal{O}$ be the order of conductor $N$ in $K$. Here we follow
the notations stated in (\ref{P_S}) and (\ref{munu}).

\begin{lemma}\label{norm}
Let $C_0$ be the identity class of $\mathrm{Cl}(\mathfrak{f})$. Then
we have
\begin{equation*}
\mathrm{N}_{K_\mathfrak{f}/H_\mathcal{O}}(g_\mathfrak{f}(C_0))^{\min\{2,N-1\}}
=\nu_N^{12N}\prod_{S\subseteq\{1,\ldots,m\}}\Delta((N/P_S)\tau_K)^{(-1)^{|S|}N},
\end{equation*}
which is an algebraic integer.
\end{lemma}
\begin{proof}
See \cite[Theorem 4.2]{E-K-S}.
\end{proof}

\begin{theorem}\label{main}
Assume that the order of the group
\begin{equation*}
\pi_{p_k^{r_k}\mathcal{O}_K}(\mathcal{O}_K)^\times/
\pi_{p_k^{r_k}\mathcal{O}_K}(\mathcal{O}_K^\times)
\pi_{p_k^{r_k}\mathcal{O}_K}(\mathbb{Z})^\times
\end{equation*}
is greater than $2$ for each $k\in\{1,\ldots,m\}$. Then the special value
\begin{equation}\label{ringinvariant}
\nu_N^{12\mu_N/\gcd(24,P_N)}\prod_{S\subseteq\{1,\ldots,m\}}\eta((N/P_S)\tau_K)^{24(-1)^{|S|}\mu_N/\gcd(24,P_N)}
\end{equation}
generates $H_\mathcal{O}$ over $K$ as a real algebraic integer.
\end{theorem}
\begin{proof}
Let $C_0$ be the identity class of $\mathrm{Cl}(\mathfrak{f})$ and
$\varepsilon=\mathrm{N}_{K_\mathfrak{f}/H_\mathcal{O}}(g_\mathfrak{f}(C_0))^\ell$
for a nonzero integer $\ell$. Suppose that $F=K(\varepsilon)$ is
properly contained in $H_\mathcal{O}$, so we can take a class
$C\in\mathrm{Cl}(K_\mathfrak{f}/F)\setminus\mathrm{Cl}(K_\mathfrak{f}/H_\mathcal{O})$.
By Lemma \ref{constructcharacter} we know that there exists a
character $\chi$ of $\mathrm{Cl}(\mathfrak{f})$ for which $\chi$ is
trivial on $\mathrm{Cl}(K_\mathfrak{f}/H_\mathcal{O})$,
$\chi(C)\neq1$ and every prime ideal factor of $\mathfrak{f}$
divides the conductor $\mathfrak{f}_\chi$ of $\chi$. Then we see
from Proposition \ref{Kronecker} and Remark \ref{Stickremark} that
the Stickelberger element $S_\mathfrak{f}(\overline{\chi})$ is
nonzero.
\par
On the other hand, we derive that
\begin{eqnarray*}
S_\mathfrak{f}(\overline{\chi})&=&\sum_{\begin{smallmatrix}C_1\in\mathrm{Cl}(\mathfrak{f})\\
C_1~\mathrm{mod}~\mathrm{Cl}(K_\mathfrak{f}/F)
\end{smallmatrix}}\sum_{\begin{smallmatrix}
C_2\in\mathrm{Cl}(K_\mathfrak{f}/F)\\
C_2~\mathrm{mod}~\mathrm{Cl}(K_\mathfrak{f}/H_\mathcal{O})
\end{smallmatrix}}\sum_{C_3\in\mathrm{Cl}(K_\mathfrak{f}/H_\mathcal{O})}\overline{\chi}(C_1C_2C_3)\ln|g_\mathfrak{f}(C_1C_2C_3)|\\
&=&\sum_{C_1}\overline{\chi}(C_1)\sum_{C_2}\overline{\chi}(C_2)\sum_{C_3}\overline{\chi}(C_3)
\ln|g_\mathfrak{f}(C_0)^{\sigma_\mathfrak{f}(C_1)\sigma_\mathfrak{f}(C_2)\sigma_\mathfrak{f}(C_3)}|
\quad\textrm{by (\ref{transformation})}\\
&=&\sum_{C_1}\overline{\chi}(C_1)\sum_{C_2}\overline{\chi}(C_2)\sum_{C_3}\ln|g_\mathfrak{f}(C_0)^{\sigma_\mathfrak{f}(C_1)\sigma_\mathfrak{f}(C_2)\sigma_\mathfrak{f}(C_3)}|
\quad\textrm{since $\chi$ is trivial on $\mathrm{Cl}(K_\mathfrak{f}/H_\mathcal{O})$}\\
&=&\sum_{C_1}\overline{\chi}(C_1)\sum_{C_2}\overline{\chi}(C_2)\ln|\mathrm{N}_{K_\mathfrak{f}/H_\mathcal{O}}(g_\mathfrak{f}(C_0))^{\sigma_\mathfrak{f}(C_1)\sigma_\mathfrak{f}(C_2)}|\\
&=&(1/\ell)\sum_{C_1}\overline{\chi}(C_1)\ln|\varepsilon^{\sigma_\mathfrak{f}(C_1)}|
\sum_{C_2}\overline{\chi}(C_2)\quad\textrm{by the fact}~\varepsilon=
\mathrm{N}_{K_\mathfrak{f}/H_\mathcal{O}}(g_\mathfrak{f}(C_0))^\ell\in F\\
&=&0,
\end{eqnarray*}
because $\chi$ can be viewed as a nonprincipal character of
$\mathrm{Cl}(K_\mathfrak{f}/F)/\mathrm{Cl}(K_\mathfrak{f}/H_\mathcal{O})$.
This yields a contradiction, and hence we achieve $
H_\mathcal{O}=F=K(\mathrm{N}_{K_\mathfrak{f}/H_\mathcal{O}}(g_\mathfrak{f}(C_0))^\ell)$.
\par
Finally, the special value in (\ref{ringinvariant})
becomes a generator of $H_\mathcal{O}$ over $K$ as a real algebraic integer
by Lemmas \ref{liein} and \ref{norm}. This completes the proof.
\end{proof}

\begin{remark}\label{condition}
Let $\mathfrak{g}=p^r\mathcal{O}_K$ for a prime $p$ and a positive integer $r$, and consider the group
\begin{equation*}
G=\pi_\mathfrak{g}(\mathcal{O}_K)^\times/
\pi_\mathfrak{g}(\mathcal{O}_K^\times)\pi_\mathfrak{g}(\mathbb{Z})^\times.
\end{equation*}
We have the order formulas
\begin{eqnarray*}
|\pi_\mathfrak{g}(\mathcal{O}_K)^\times|&=&
p^{2(r-1)}(p-1)\bigg(p-\bigg(\frac{d_K}{p}\bigg)\bigg),\\
|\pi_\mathfrak{g}(\mathcal{O}_K^\times)\pi_\mathfrak{g}(\mathbb{Z})^\times|&=&
(|\mathcal{O}_K^\times|/2)p^{r-1}(p-1),
\end{eqnarray*}
where $(\frac{d_K}{p})$ stands for the Kronecker symbol.
\begin{equation*}
\bigg(\frac{d_K}{p}\bigg)=\left\{\begin{array}{ll}
\textrm{the Legendre symbol} & \textrm{if $p$ is odd},\\
\textrm{the Kronecker symbol} & \textrm{if}~p=2
\end{array}\right.
\end{equation*}
\cite[p.148]{Cox}. Then we achieve
\begin{equation*}
|G|=\frac{2p^{r-1}}{|\mathcal{O}_K^\times|}\bigg(p-\bigg(\frac{d_K}{p}\bigg)\bigg).
\end{equation*}
And, one can classify all the cases in which $|G|=1$ or $2$ as listed in Table \ref{table1}:
\begin{table}[!h]
\begin{flushleft}
\footnotesize
\begin{tabular} {|c||c|c|c|c|c|c|}
\hline
\multirow{2}{10mm}{$~~~K$} &
\multicolumn{6}{|c|}{neither
$\mathbb{Q}(\sqrt{-1})$ nor $\mathbb{Q}(\sqrt{-3})$}\\
\cline{2-7}
 & \tiny$d_K\equiv 1\pmod{24}$ &
 \tiny$d_K\equiv 9,17\pmod{24}$ &
 \tiny$d_K\equiv 4,16\pmod{24}$ &
 \tiny$d_K\equiv 0,8,12,20\pmod{24}$ &
 \tiny$d_K\equiv 13\pmod{24}$ &
 \tiny otherwise\\
\hline\hline
$~~~p^r$ &  $2,2^2,3$ & $2,2^2$ & $2,3$ & $2$ & $3$ & none\\ 
\hline
\end{tabular}
\vspace{0.2cm}\\
\begin{tabular} {|c||c|c|}
\hline
{$\quad K\quad$} &
{$\mathbb{Q}(\sqrt{-1})$} &
{$\mathbb{Q}(\sqrt{-3})$} 
\\
\hline\hline
$~~~p^r$ & $2,2^2,3,5$ & $2,2^2,3,5,7$ \\
\hline
\end{tabular}
\end{flushleft}
\caption{The cases $|G|=1$ or $2$}\label{table1}
\end{table}
\end{remark}

\section {Examples}

In this last section, we shall present some examples
of computing
minimal polynomials of ring class invariants
with relatively small coefficients
developed in Theorem \ref{main}
so as to apply them to quadratic Diophantine equations
concerning non-convenient numbers.

\begin{example}\label{Example1}
Let $K=\mathbb{Q}(\sqrt{-1})$, so $d_K=-4$ and $\tau_K=\sqrt{-1}$.
In this case, we have $H_K=K$.
Let $\mathcal{O}$ be the order of conductor $13$ in $K$. Then
the special value
$13(\eta(13\tau_K)/\eta(\tau_K))^2$ generates $H_\mathcal{O}$ over
$K$ as a real algebraic integer by Theorem \ref{main} and Remark
\ref{condition}. And, we attain by Remark \ref{Galoisgroup} that
\begin{eqnarray*}
\mathrm{Gal}(H_\mathcal{O}/K)&\simeq&W_{K,13}
/\langle T_{K,13},tI_2~|~t\in(\mathbb{Z}/13\mathbb{Z})^\times\rangle\\
&=&\{\left[\begin{smallmatrix}1&0\\0&1\end{smallmatrix}\right]
\left[\begin{smallmatrix}1&0\\0&1\end{smallmatrix}\right],
\left[\begin{smallmatrix}1&0\\0&2\end{smallmatrix}\right]
\left[\begin{smallmatrix}1&-1\\-6&7\end{smallmatrix}\right],
\left[\begin{smallmatrix}1&0\\0&5\end{smallmatrix}\right]
\left[\begin{smallmatrix}1&-2\\3&-5\end{smallmatrix}\right],\\
&&
\left[\begin{smallmatrix}1&0\\0&10\end{smallmatrix}\right]
\left[\begin{smallmatrix}1&-3\\-1&4\end{smallmatrix}\right],
\left[\begin{smallmatrix}1&0\\0&11\end{smallmatrix}\right]
\left[\begin{smallmatrix}14&-33\\3&-7\end{smallmatrix}\right],
\left[\begin{smallmatrix}1&0\\0&4\end{smallmatrix}\right]
\left[\begin{smallmatrix}1&4\\-1&-3\end{smallmatrix}\right]
\},
\end{eqnarray*}
from which we can compute (by using Maple ver.15)
\begin{eqnarray*}
&&\min(13(\eta(13\tau_K)/\eta(\tau_K))^2,K)\\&=&
(X-13(\eta(13\tau)/\eta(\tau))^2\circ\left[\begin{smallmatrix}1&0\\0&1\end{smallmatrix}\right]
(\tau_K))
(X-13(\eta(13\tau)/\eta(\tau))^2\circ\left[\begin{smallmatrix}1&-1\\-6&7\end{smallmatrix}\right]
(\tau_K))
\\
&&(X-13(\eta(13\tau)/\eta(\tau))^2\circ\left[\begin{smallmatrix}1&-2\\3&-5\end{smallmatrix}\right]
(\tau_K))
(X-13(\eta(13\tau)/\eta(\tau))^2\circ\left[\begin{smallmatrix}1&-3\\-1&4\end{smallmatrix}\right]
(\tau_K))\\
&&(X-13(\eta(13\tau)/\eta(\tau))^2\circ\left[\begin{smallmatrix}14&-33\\3&-7\end{smallmatrix}\right]
(\tau_K))
(X-13(\eta(13\tau)/\eta(\tau))^2\circ\left[\begin{smallmatrix}1&4\\-1&-3\end{smallmatrix}\right]
(\tau_K))\\
&=&X^6+10X^5+46X^4+108X^3+122X^2+38X-1.
\end{eqnarray*}
On the other hand, by using the relation
\begin{equation*}
j(\tau)=(2^8\eta(2\tau)^{16}\eta(\tau)^{-16}+\eta(2\tau)^{-8}\eta(\tau)^8)^3
\end{equation*}
\cite[pp.256--257]{Cox}, one can also get
\begin{eqnarray*}
\min(j(13\tau_K),K)&=&
X^6
-10368X^5
+44789760X^4
-103195607040X^3\\
&&+133741506723840X^2
-92442129447518208X\\
&&+26623333280885243904.
\end{eqnarray*}
Here we observe that the coefficients of $\min(13(\eta(13\tau_K)/\eta(\tau_K))^2,K)$
are relatively smaller than those of
$\min(j(13\tau_K),K)$.
\end{example}

\begin{example}\label{Example2}
Let $K=\mathbb{Q}(\sqrt{-7})$, so $d_K=-7$ and $\tau_K=(-1+\sqrt{-7})/2$.
Then we know $H_K=K$, too.
\begin{itemize}
\item[(i)] Let $\mathcal{O}$ be the order of conductor $7$ in $K$.
Then the special value
$7^2(\eta(7\tau_K)/\eta(\tau_K))^4$ generates $H_\mathcal{O}$ over
$K$ as a real algebraic integer by Theorem \ref{main} and Remark
\ref{condition}. And, we derive by Remark \ref{Galoisgroup}
\begin{eqnarray*}
\mathrm{Gal}(H_\mathcal{O}/K)&\simeq&
W_{K,7}/\langle T_{K,7},tI_2~|~t\in(\mathbb{Z}/7\mathbb{Z})^\times\rangle\\
&=&\{\left[\begin{smallmatrix}1&0\\0&1\end{smallmatrix}\right]
\left[\begin{smallmatrix}1&0\\0&1\end{smallmatrix}\right],
\left[\begin{smallmatrix}1&0\\0&2\end{smallmatrix}\right]
\left[\begin{smallmatrix}-1&-2\\4&7\end{smallmatrix}\right],
\left[\begin{smallmatrix}1&0\\0&2\end{smallmatrix}\right]
\left[\begin{smallmatrix}7&-2\\11&-3\end{smallmatrix}\right],\\
&&\left[\begin{smallmatrix}1&0\\0&4\end{smallmatrix}\right]
\left[\begin{smallmatrix}1&5\\2&11\end{smallmatrix}\right],
\left[\begin{smallmatrix}1&0\\0&1\end{smallmatrix}\right]
\left[\begin{smallmatrix}2&5\\1&3\end{smallmatrix}\right],
\left[\begin{smallmatrix}1&0\\0&1\end{smallmatrix}\right]
\left[\begin{smallmatrix}-3&5\\1&-2\end{smallmatrix}\right],
\left[\begin{smallmatrix}1&0\\0&4\end{smallmatrix}\right]
\left[\begin{smallmatrix}5&12\\2&5\end{smallmatrix}\right]\}
\end{eqnarray*}
and obtain
\begin{eqnarray*}
&&\min(7^2(\eta(7\tau_K)/\eta(\tau_K))^4,K)\\&=&
(X-7^2(\eta(7\tau)/\eta(\tau))^4\circ\left[\begin{smallmatrix}1&0\\0&1\end{smallmatrix}\right]
(\tau_K))
(X-7^2(\eta(7\tau)/\eta(\tau))^4\circ\left[\begin{smallmatrix}-1&-2\\4&7\end{smallmatrix}\right]
(\tau_K))\\
&&
(X-7^2(\eta(7\tau)/\eta(\tau))^4\circ\left[\begin{smallmatrix}7&-2\\11&-3\end{smallmatrix}\right]
(\tau_K))
(X-7^2(\eta(7\tau)/\eta(\tau))^4\circ\left[\begin{smallmatrix}1&5\\2&11\end{smallmatrix}\right]
(\tau_K))\\
&&(X-7^2(\eta(7\tau)/\eta(\tau))^4\circ\left[\begin{smallmatrix}2&5\\1&3\end{smallmatrix}\right]
(\tau_K))
(X-7^2(\eta(7\tau)/\eta(\tau))^4\circ\left[\begin{smallmatrix}-3&5\\1&-2\end{smallmatrix}\right]
(\tau_K))\\
&&(X-7^2(\eta(7\tau)/\eta(\tau))^4\circ\left[\begin{smallmatrix}5&12\\2&5\end{smallmatrix}\right]
(\tau_K))\\
&=&X^7+21X^6+175X^5+679X^4+1162X^3+490X^2+588X+7.
\end{eqnarray*}
\item[(ii)] Now, let $\mathcal{O}$ be the order of conductor $6$.
We know by Remark \ref{Galoisgroup} that
\begin{eqnarray*}
\mathrm{Gal}(H_\mathcal{O}/K)&\simeq&
W_{K,6}/\langle T_{K,6},tI_2~|~t\in(\mathbb{Z}/6\mathbb{Z})^\times\rangle\\
&=&\{\left[\begin{smallmatrix}1&0\\0&1\end{smallmatrix}\right]
\left[\begin{smallmatrix}1&0\\0&1\end{smallmatrix}\right],
\left[\begin{smallmatrix}1&0\\0&1\end{smallmatrix}\right]
\left[\begin{smallmatrix}5&2\\2&1\end{smallmatrix}\right],
\left[\begin{smallmatrix}1&0\\0&-1\end{smallmatrix}\right]
\left[\begin{smallmatrix}3&-2\\2&-1\end{smallmatrix}\right],
\left[\begin{smallmatrix}1&0\\0&-1\end{smallmatrix}\right]
\left[\begin{smallmatrix}1&2\\-2&-3\end{smallmatrix}\right]\},
\end{eqnarray*}
which is of order $4$.
Set $h(\tau)=(\eta(6\tau)\eta(\tau)/\eta(3\tau)\eta(2\tau))^{12}$.
Since $h(\tau_K)\in\mathbb{R}$, we deduce that
\begin{eqnarray*}
[K(h(\tau_K)):K]&=&
[K(h(\tau_K)):\mathbb{Q}]/[K:\mathbb{Q}]\\
&=&
[K(h(\tau_K)):\mathbb{Q}(h(\tau_K))][\mathbb{Q}(h(\tau_K)):\mathbb{Q}]
/[K:\mathbb{Q}]\\
&=&[\mathbb{Q}(h(\tau_K)):\mathbb{Q}].
\end{eqnarray*}
Furthermore, we see that the polynomial
\begin{eqnarray*}
&&\prod_{\gamma\in\mathrm{Gal}(H_\mathcal{O}/K)}(X-h(\tau_K)^\gamma)
\\&=&(X-h(\tau)\circ\left[\begin{smallmatrix}
1&0\\0&1
\end{smallmatrix}\right](\tau_K))
(X-h(\tau)\circ\left[\begin{smallmatrix}
5&2\\2&1
\end{smallmatrix}\right](\tau_K))\\
&&(X-h(\tau)\circ\left[\begin{smallmatrix}
3&-2\\2&-1
\end{smallmatrix}\right](\tau_K))
(X-h(\tau)\circ\left[\begin{smallmatrix}
1&2\\-2&-3
\end{smallmatrix}\right](\tau_K))\\
&=&X^4-35X^3+198X^2+4060X+1
\end{eqnarray*}
is irreducible over $\mathbb{Q}$. Hence
$h(\tau_K)=(\eta(6\tau_K)\eta(\tau_K)/\eta(3\tau_K)\eta(2\tau_K))^{12}$
generates $H_\mathcal{O}$ over $K$ as a unit, although we couldn't
directly apply Theorem \ref{main} to this case. This example
indicates that there seems to be a room for the condition on the
order of the group
$\pi_{p_k^{r_k}\mathcal{O}_K}(\mathcal{O}_K)^\times/
\pi_{p_k^{r_k}\mathcal{O}_K}(\mathcal{O}_K^\times)
\pi_{p_k^{r_k}\mathcal{O}_K}(\mathbb{Z})^\times$ to be improved
further.
\end{itemize}
\end{example}

\begin{example}\label{Example3}
Let $K=\mathbb{Q}(\sqrt{-6})$. We then have $d_K=-24$ and
$\tau_K=\sqrt{-6}$. Let $\mathcal{O}$ be the order of conductor $3$
in $K$. Then the real algebraic integer
$3^6(\eta(3\tau_K)/\eta(\tau_K))^{12}$ generates $H_\mathcal{O}$
over $K$ by Theorem \ref{main} and Remark \ref{condition}. And, we
see by Remark \ref{Galoisgroup} that
\begin{equation*}
\mathrm{Gal}(H_\mathcal{O}/H_K)\simeq W_{K,3}/\langle
T_{K,3},tI_2~|~t\in(\mathbb{Z}/3\mathbb{Z})^\times\rangle
=\{\left[\begin{smallmatrix}1&0\\0&1\end{smallmatrix}\right],
\left[\begin{smallmatrix}1&0\\1&1\end{smallmatrix}\right],
\left[\begin{smallmatrix}1&0\\2&1\end{smallmatrix}\right]\}.
\end{equation*}
However, in this case, $H_K\neq K$. On the other hand, as is
well-known $\mathrm{Gal}(H_K/K)$ is isomorphic to the form class
group $\mathrm{C}(d_K)$ of discriminant $d_K=-24$ which consists of
two reduced primitive positive definite quadratic forms
\begin{equation*}
Q_1=X^2+6Y^2\quad\textrm{and}\quad
Q_2=2X^2+3Y^2
\end{equation*}
\cite[Theorems 2.8, 5.23 and 7.7]{Cox}.
Corresponding to $Q_1$ and $Q_2$ we let
\begin{equation*}
\beta_1=\left[\begin{smallmatrix}1&0\\0&1\end{smallmatrix}\right],\tau_1=\sqrt{-6}
\quad\textrm{and}\quad
\beta_2=\left[\begin{smallmatrix}2&0\\0&1\end{smallmatrix}\right],\tau_2=\sqrt{-6}/2,
\end{equation*}
respectively. Then due to Stevenhagen \cite{Stevenhagen} the Galois
conjugates of $h(\tau_K)$, where
$h(\tau)=3^6(\eta(3\tau)/\eta(\tau))^{12}$, are given by
\begin{equation*}
h^{\gamma\beta_k}(\tau_k)\quad\textrm{for}~\gamma\in\mathrm{Gal}(H_\mathcal{O}/H_K)~\textrm{and}~k=1,2
\end{equation*}
(see also \cite[Theorem 2.4]{J-K-S}). And, we achieve
\begin{eqnarray*}
&&\{\gamma\beta_k~|~\gamma\in\mathrm{Gal}(H_\mathcal{O}/H_K),~k=1,2\}
\quad(\subseteq\mathrm{GL}_2(\mathbb{Z}/3\mathbb{Z})/\{\pm I_2\})\\&=&
\{\left[\begin{smallmatrix}1&0\\0&1\end{smallmatrix}\right]
\left[\begin{smallmatrix}1&0\\0&1\end{smallmatrix}\right],
\left[\begin{smallmatrix}1&0\\0&1\end{smallmatrix}\right]
\left[\begin{smallmatrix}1&0\\1&1\end{smallmatrix}\right],
\left[\begin{smallmatrix}1&0\\0&1\end{smallmatrix}\right]
\left[\begin{smallmatrix}1&0\\2&1\end{smallmatrix}\right],
\left[\begin{smallmatrix}1&0\\0&2\end{smallmatrix}\right]
\left[\begin{smallmatrix}-1&0\\0&-1\end{smallmatrix}\right],
\left[\begin{smallmatrix}1&0\\0&2\end{smallmatrix}\right]
\left[\begin{smallmatrix}-1&0\\1&-1\end{smallmatrix}\right],
\left[\begin{smallmatrix}1&0\\0&2\end{smallmatrix}\right]
\left[\begin{smallmatrix}-1&0\\2&-1\end{smallmatrix}\right]\},
\end{eqnarray*}
from which we get
\begin{eqnarray*}
&&\min(3^6(\eta(3\tau_K)/\eta(\tau_K))^{12},K)\\
&=&
(X-h(\tau)\circ\left[\begin{smallmatrix}1&0\\0&1\end{smallmatrix}\right](\sqrt{-6}))
(X-h(\tau)\circ\left[\begin{smallmatrix}1&0\\1&1\end{smallmatrix}\right](\sqrt{-6}))
(X-h(\tau)\circ\left[\begin{smallmatrix}1&0\\2&1\end{smallmatrix}\right](\sqrt{-6}))\\
&&(X-h(\tau)\circ\left[\begin{smallmatrix}-1&0\\0&-1\end{smallmatrix}\right](\sqrt{-6}/2))
(X-h(\tau)\circ\left[\begin{smallmatrix}-1&0\\1&-1\end{smallmatrix}\right](\sqrt{-6}/2))\\
&&(X-h(\tau)\circ\left[\begin{smallmatrix}-1&0\\2&-1\end{smallmatrix}\right](\sqrt{-6}/2))\\
&=&X^6+234X^5+39015X^4+1335852X^3+14036895X^2-4833270X+729.
\end{eqnarray*}
\end{example}

\begin{remark}\label{quadratic}
Let $n$ be a positive integer and $f_n(X)\in\mathbb{Z}[X]$ be the
minimal polynomial of a real algebraic integer which generates the
ring class field of the order $\mathbb{Z}[\sqrt{-n}]$ in the
imaginary quadratic field $\mathbb{Q}(\sqrt{-n})$. Then we have the
assertion that if an odd prime $p$ divides neither $n$ nor the
discriminant of $f_n(X)$, then $p$ can be written in the form
$p=x^2+ny^2$ for some $x,y\in\mathbb{Z}$ if and only if
$(\frac{-n}{p})=1$ and $f_n(X)\equiv0\pmod{p}$ has an integer
solution as well \cite[Theorem 9.2]{Cox}. Whenever
the equivalent condition cannot be expressed as
$p\equiv c_1,\ldots,c_m\pmod{4n}$, we call such $n$ a non-convenient number. As for the
convenient numbers we refer to \cite[$\S$3.C]{Cox}, \cite{Weil} and \cite{Weinberger}.
\begin{itemize}
\item[(i)]
We are able to use our ring class invariants in Examples
\ref{Example1} and \ref{Example3} for these quadratic Diophantine
problems. First, let $K=\mathbb{Q}(\sqrt{-1})$ (, so
$\tau_K=\sqrt{-1}$). Then we get
\begin{equation*}
\mathrm{disc}(13(\eta(13\tau_K)/\eta(\tau_K))^2,K))
=2^{10}\cdot3^6\cdot13^5,
\end{equation*}
and derive that a prime $p$ satisfies $(\frac{-169}{p})=1$ if and only
if $p=2$ or $p\equiv1\pmod{4}$. Thus we reach the conclusion that
if $p$ is a prime other than $13$, then $p$ can be written in the form
$p=x^2+169y^2$ for some $x,y\in\mathbb{Z}$
if and only if $p\equiv1\pmod{4}$ and
$X^6+10X^5+46X^4+108X^3+122X^2+38X-1\equiv0\pmod{p}$ has an integer solution.
\par
Second, let $K=\mathbb{Q}(\sqrt{-6})$ (, so $\tau_K=\sqrt{-6}$). We compute
\begin{equation*}
\mathrm{disc}(3^6(\eta(3\tau_K)/\eta(\tau_K))^{12},K))
=2^{69}\cdot3^{36}\cdot13^4\cdot17^2\cdot19^4\cdot23^2,
\end{equation*}
and find that a prime $p$ satisfies $(\frac{-54}{p})=1$ if and only if
$p\equiv1,5,7,11\pmod{24}$.
Thus we can conclude that a prime $p$ can be written in the form $p=x^2+54y^2$ for some $x,y\in\mathbb{Z}$ if and only if $p\equiv1,5,7,11\pmod{24}$ and $X^6+234X^5+39015X^4+1335852X^3+14036895X^2-4833270X+729\equiv0\pmod{p}$ has an integer solution.

\item[(ii)] In like manner, one can further show by using Example
\ref{Example2}(ii) that
 a prime $p$ can be expressed as $p=x^2+63y^2$ for some $x,y\in\mathbb{Z}$
 if and only if $p\equiv1,9,11\pmod{14}$ and $X^4-35X^3+198X^2+4060X+1\equiv0\pmod{p}$ has an integer solution.
\end{itemize}
\end{remark}

\bibliographystyle{amsplain}

\address{
Department of Mathematical Sciences \\
KAIST \\
Daejeon 34141\\
Republic of Korea} {jkkoo@math.kaist.ac.kr}
\address{
Department of Mathematics\\
Hankuk University of Foreign Studies\\
Yongin-si, Gyeonggi-do 17035\\
Republic of Korea} {dhshin@hufs.ac.kr}
\address{
Department of Mathematical Sciences \\
KAIST \\
Daejeon 34141\\
Republic of Korea} {math\_dsyoon@kaist.ac.kr}

\end{document}